\documentclass[12pt]{elsarticle}
\usepackage{amsfonts,amsmath,oldgerm,amssymb,amscd}
\usepackage{lineno}

\usepackage{amsthm}
\newtheorem{theorem}{Theorem}[section]
\newtheorem{lemma}[theorem]{Lemma}
\newtheorem{corollary}[theorem]{Corollary}
\newtheorem{proposition}[theorem]{Proposition}

\usepackage{chngcntr}
\newdefinition{remark}{Remark}
\newproof{pf}{Proof}
\theoremstyle{definition}
\newtheorem{defi}{Definition}
\newtheorem{notation}[theorem]{Notation}
\newtheorem{example}{Example}
\usepackage{mathtools}
\usepackage{amsfonts}
\newcommand{\field}[1]{\mathbb{#1}}          \newcommand{\Q}{\field{Q}}
\newcommand{\N}{\field{N}}
\newcommand{\R}{\field{R}}                   \newcommand{\Z}{\field{Z}}
\newcommand{\C}{\field{C}}

\newcommand{\im}{\rm Im}

\newcommand{\mF}{{\mathcal F}}
\newcommand{\mX}{{\mathcal X}}

\newcommand{\f}{{\mathcal F}}

\newcommand{\ep}{\epsilon}

\renewcommand{\a}{\alpha}                      \renewcommand{\b}{\beta}

\newcommand{\sign}{\textnormal{sign}}

\clearpage
\begin{document}
	\title{ Best Approximations by $\mathcal{F}_{p^l}$-Continued Fractions} 
	
	\author[sk]{S. Kushwaha}\corref{cor1}  
	\ead{seema28k@gmail.com}
	\author[rvt]{R. Sarma} 
	\ead{ritumoni@maths.iitd.ac.in}

	\cortext[cor1]{Corresponding author}
	\address[sk]{Department of Applied Sciences, Indian Institute of Information Technology Allahabad, Prayagraj, India}
	\address[rvt]{Department of Mathematics,
		Indian Institute of Technology Delhi, India}

	\begin{abstract}
		
		In this article, for a certain subset $\mathcal{X}$ of the extended set of rational numbers, we introduce the notion of {\it best $\mathcal{X}$-approximations} of a real number.  The notion of best $\mathcal{X}$-approximation  is analogous to that of best rational approximation.   We explore these approximations with the help of $\mathcal{F}_{p^l}$-continued fractions, where $p$ is a prime and $l\in\mathbb{N}$, we show that the convergents of the $\mathcal{F}_{p^l}$-continued fraction expansion  of a real number $x$ satisfying certain maximal conditions are exactly the best $\mathcal{F}_{p^l}$-approximations of $x$.
	
	\end{abstract}
	\maketitle	
	\section{Introduction}

Let $p$ be a prime and $l\in \N.$	A finite continued fraction of the form $$\frac{1}{0+}~\frac{{p^l}}{b+}~\frac{\epsilon_{1} }{a_{1}+}~\frac{\epsilon_{2}}{a_{2}+}~\cdots\frac{\epsilon_{n}}{a_{n}}~~(n\ge 0)$$
or
an infinite continued fraction of the form
$$\frac{1}{0+}~\frac{{p^l}}{b+}~\frac{\epsilon_{1} }{a_{1}+}~\frac{\epsilon_{2}}{a_{2}+}~\cdots\frac{\epsilon_{n}}{a_{n}+}\cdots$$ 
is called an $\mF_{{p^l}}$-{\it continued fraction},
 where $b$ is an integer co-prime to $p$ and for $i\ge1$,  $a_i\in\N$ and  $\ep_i\in\{\pm1\}$  with certain conditions.  A precise definition of an $\f_{{p^l}}$-continued fraction is stated in Section 3 introduced by Kushwaha et al. \cite{seemafnpart1}.
  In fact, this family of continued fractions  arises from a family of graphs $\f_{{p^l}}$ which are similar to the Farey graph. The value of a finite $\f_{p^l}$-continued fraction is a member of the set 
 \begin{equation}\label{X_n}
 		\mathcal{X}_{p^l}=\left\{\frac{x}{y}:~x,y\in\mathbb{Z},~ y>0,~\mathrm{gcd}(x,y)=1~\textnormal{and}~{p^l}|y\right\}\cup\{\infty\}
\end{equation}   
 which is the vertex set of $\f_{p^l}$.
 The important fact is that every real number has an $\f_{p^l}$-continued fraction expansion.

 An element  $u/v$ of $\mX_{p^l}$ is called a \textit{best $\mX_{{p^l}}$-approximation} of $x\in\R$, if for every $u'/v'\in\mX_{{p^l}}$ different from $u/v$ with	$0< v' \le v$, we have $|vx-u|<|v'x-u'|$.  

These approximations have been discussed in \cite{seema2,seema} for $p^l=3,2$ respectively. In these papers, authors have achieved results analogous to the classical one, that is,  convergents of the continued fraction of a real number characterize the best approximations of the real number.

 In several recent papers \cite{lucax_tribnocci1,lucax_twofibonacci,primepowers,rep}, the following problem was investigated. Let $U$
be some interesting set of positive integers. What can one say about the
square-free integers $d$ such that the  first (or, the second) coordinate $X\,(\textnormal{respectively, }Y)$ of a solution to the Pell equation $X^2-dY^2=1$ is a member of the set $U$. The first author has applied best $\mX_{2^l}$-approximations to solve certain conditional Pell equations \cite{seema_pell} which is a special case of the above mentioned problem. We strongly believe that a complete generalization of work in \cite{seema2,seema} will be very useful and in this article, we deal with the relation of best $\mX_{p^l}$-approximations and $\f_{p^l}$-convergents.

Note that a real number may have more than one $\f_{p^l}$-continued fraction expansions; this fact was observed in \cite{seema2} for $p^l=3$. 
For instance, the set of $\f_5$-continued fraction expansions of $11/40$ is as follows
\begin{align*}
&	\Big\{\frac{1}{0+}~\frac{5}{1+}~\frac{1}{2+}~\frac{1}{1+}~\frac{1}{1+}~\frac{1}{1},~ \frac{1}{0+}~\frac{5}{1+}~\frac{1}{2+}~\frac{1}{2+}~\frac{-1}{2},~\\
&\frac{1}{0+}~\frac{5}{1+}~\frac{1}{2+}~\frac{1}{1+}~\frac{1}{2},
\frac{1}{0+}~\frac{5}{1+}~\frac{1}{3+}~\frac{-1}{2+}~\frac{1}{1},~\frac{1}{0+}~\frac{5}{1+}~\frac{1}{3+}~\frac{-1}{3},~\\ \tag*{(2)}\label{manyexpansions}
&\frac{1}{0+}~\frac{5}{2+}~\frac{-1}{2+}~\frac{-1}{2+}~\frac{1}{2}, \frac{1}{0+}~\frac{5}{2+}~\frac{-1}{2+}~\frac{-1}{3+}~\frac{-1}{2},
\frac{1}{0+}~\frac{5}{2+}~\frac{-1}{2+}~\frac{-1}{2+}~\frac{1}{1+}~\frac{1}{1}
\Big\}.
\end{align*}
When $p=3$, the longest $\f_{p^l}$-continued fraction is unique (see \cite{seema2}) and this was helpful to achieve the approximation results. But this is false for $p\ge5.$ In the above example,  $ \frac{1}{0+}~\frac{5}{1+}~\frac{1}{2+}~\frac{1}{1+}~\frac{1}{1+}~\frac{1}{1}$ and $	\frac{1}{0+}~\frac{5}{2+}~\frac{-1}{2+}~\frac{-1}{2+}~\frac{1}{1+}~\frac{1}{1} $ are both longest $\f_5$-continued fractions of 11/40.  Besides, not every longest $\f_{p^l}$-continued fraction is helpful to describe best $\mX_{p^l}$-approximations. In fact, best approximations are described by $\f_{p^l}$-continued fractions with maximum $+1$ (see Theorem 4.9 and 4.10).

	Finally, we show that for any real number $x$ which is not in $\Q\setminus\mX_{p^l},$ every convergent of the $\f_{p^l}$-continued fraction of $x$ with maximum $+1$ is a best $\mX_{p^l}$-approximation of $x$  and conversely. 
	
	\section{Preliminaries}

In this section, we recall certain definitions and results on $\f_{p^l}$-continued fractions from \cite{seemafnpart1}. Now onwards, $N$ denotes a positive integer of the form $p^l,$ where $p$ is a prime and $l$ is a natural number.
\begin{defi}
Given $N=p^l,$	a finite continued fraction of the form $$\frac{1}{0+}~\frac{N}{b+}~\frac{\epsilon_{1} }{a_{1}+}~\frac{\epsilon_{2}}{a_{2}+}~\cdots\frac{\epsilon_{n}}{a_{n}}~~(n\ge 0)$$
	or
	an infinite continued fraction of the form
	$$\frac{1}{0+}~\frac{N}{b+}~\frac{\epsilon_{1} }{a_{1}+}~\frac{\epsilon_{2}}{a_{2}+}~\cdots\frac{\epsilon_{n}}{a_{n}+}\cdots$$ 
	is called an $\mF_{N}$-{\it continued fraction} if $b$ is an integer co-prime to $N$, and for $i\ge1$, $a_i\in\N$ and $\ep_i\in\{\pm 1\}$ such that the following conditions hold:
	\begin{enumerate}
		\item $a_i+\ep_{i+1}\ge1$;
		\item $a_i+\ep_i\ge1$;
		\item $\mathrm{gcd}(p_i,q_i)=1,$ where $p_i=a_i p_{i-1}+\ep_i p_{i-2}$, $q_i=a_i q_{i-1}+\ep_i q_{i-2}$, $(p_{-1},q_{-1})=(1,0)$ and $(p_0,q_0)=(b,N)$. 
	\end{enumerate}  
\end{defi} 
	For $i\ge1,$ the value $p_i/q_i$ of the the expression
	$$\frac{1}{0+}~\frac{N}{b+}~\frac{\epsilon_{1} }{a_{1}+}~\frac{\epsilon_{2}}{a_{2}+}~\cdots\frac{\epsilon_{i}}{a_{i}}$$
	is called the {\it $i$-th $\f_N$-convergent} of the continued fraction. The sequence $\{\frac{p_i}{q_i}\}_{i\ge0}$ is called the {\it sequence of $\f_N$-convergents}. The expression 
	$\frac{\epsilon_{i}}{a_{i}+}~\frac{\epsilon_{i+1}}{a_{i+1}+}\cdots $ is called the {\it $i$-th fin}. Let $y_i$ denote the $i$-th fin, that is,  $y_i=\frac{\epsilon_{i}}{a_{i}+}~\frac{\epsilon_{i+1}}{a_{i+1}+}\cdots $, then $\ep_{i}=\sign(y_{i})$.
	 The following theorem assets certain properties of $\f_N$-continued fractions. 
	\begin{theorem}\cite[Theorem 3.2]{seemafnpart1}\label{distinctconvergents}
		Suppose $x= \frac{1}{0+}~\frac{N}{b+}~ \frac{\ep_1}{a_1 +}~ \frac{\ep_2}{a_2 +}~ \frac{\ep_3}{a_3 +}~\cdots
		$ is an $\f_N$-continued fraction with the sequence of convergents $\{p_i/q_i\}_{i\ge -1}$. For $i\ge1$, let  
		$y_i$ be the $i$-th \it{fin} of the continued fraction.  Then
		\begin{enumerate}
			\item 	for  $i\ge1,$ $a_ip_{i-1}\not \equiv -\ep_{i}p_{i-2}\mod p$;
			\item 	the sequence $\{q_i\}_{i\ge -1}$ is strictly increasing;
			\item $\dfrac{p_i}{q_i}\ne \dfrac{p_j}{q_j}$ for $i\ne j$;
			\item for $i\ge1$, $|y_i|\le1$;
			\item for $i\ge0,$ $x=\dfrac{x_{i+1}p_i+\ep_{i+1}p_{i-1}}{x_{i+1}q_i+\ep_{i+1}q_{i-1}},$ where $x_{i+1}=\dfrac{1}{|y_{i+1}|}.$
		\end{enumerate}
	\end{theorem}
	
	 An $\f_N$-continued fractions is arising from the graph $\f_N,$ where
 the  vertex set is  $\mathcal{X}_N$ (as defined in Equation \eqref{X_n})  and vertices ${p}/{q}$ and ${r}/{s}$, are adjacent in $\mathcal{F}_N$ if and only if  $$rq-sp=\pm N.$$ 
		If $P$ and $Q$ are adjacent in $\f_N$ we write $P\sim_N Q$. Note that the graph $\f_{1}$ is the Farey graph 
and for every $N\in\N$, the graph $\f_N$ is isomorphic to a subgraph of the Farey graph.

	 \begin{figure}[!h]
	 	\centering
	 	\includegraphics[width=.8\linewidth]{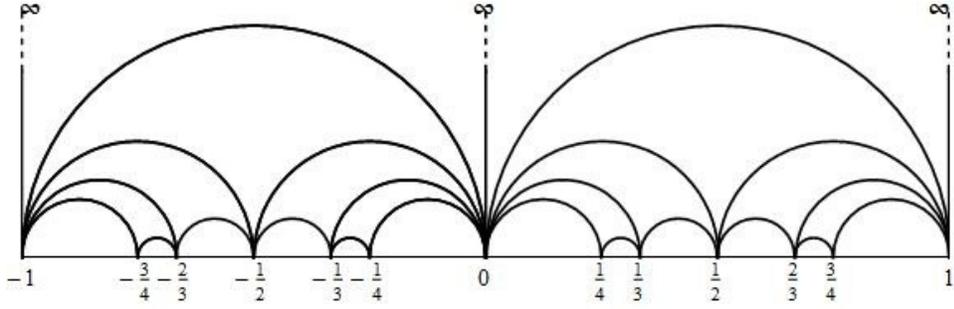}
	 	\caption{A few vertices and edges of the Farey graph in [-1,1]}\label{Fig:farey}
	 \end{figure}
	\noindent Edges of $\f_{N}$ are represented as hyperbolic geodesics in the upper-half plane $$\mathcal{U}=\{z\in\C: \im(z)>0\},$$
	that is, as Euclidean semicircles or half lines perpendicular to the real line. Figure 1 is a display of a few edges of the Farey graph in the interval [-1,1].
Since edges of the Farey graph do not cross each other, and $\f_N$ is embedded in the Farey graph, we have the following result.
	
	\begin{proposition}\cite[Corollary 2.2]{seemafnpart1}\label{nocrossing}
		No two edges cross in $\f_N$.
	\end{proposition}
Here, we recall a few definitions which will help us to show that every irrational number has a unique $\f_N$-continued fraction expansion.
 	 \begin{defi} Let $x\in\mX_N.$ Suppose $\Theta_n\equiv\infty=P_{-1}\to P_0\to P_1\to\cdots\to P_{n-1}\to P_n,$ where $x=P_n,$ is such that no vertex is repeated (i.e., $P_i\ne P_j$ for $-1\le i\ne j\le n$).  
	 Let $Q\in \mX_N$ such that $Q\ne P_i$ for $-1\le i\le n$ and $x\sim_N Q$. 
  	   If $P_{n-1}<Q<x$ or $x<Q<P_{n-1}$, then the edge $x\to Q$ is called a {\it direction changing edge} from $x$ relative to $\Theta_n$ (see Figure 2). 
  	 
  If $P_{n-1}<x<Q \textnormal{ or } P_{n-1}>x>Q$, then the edge $x\to Q$ is called a {\it direction retaining edge} from $x$ relative to $\Theta_n$ (see Figure 3).
  		
  	\end{defi}
  	  \begin{figure}[h!]
  	 	\centering
  	 	\includegraphics[scale= 1.5]{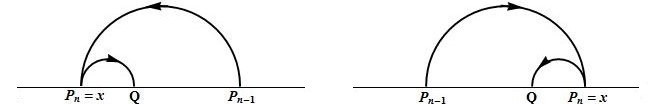}\caption{Direction changing edge $P_n\sim_N Q$}
  	 \end{figure}   
  	 \begin{figure}[h!]
  	 	\centering
  	 	  	 	\includegraphics[scale= 1.5]{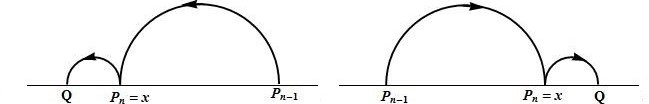}
  	 	\caption{Direction retaining edge $P_n\sim_N Q$}		
  	 \end{figure}
\begin{defi} Suppose $n\in\N.$
	A path from infinity to a vertex $x$ in $\f_N$ given by
	$$\Theta_n\equiv\infty=P_{-1}\rightarrow P_0\rightarrow P_1\rightarrow\cdots \to  P_n,$$	where $P_i=p_i/q_i$ and $q_{i}<q_{i+1}$ for $i\ge-1$ and $x=P_n$, is called a \textit{well directed path} if  $P_{i+1}\sim_NP_{i+2}$ ($0\le i\le n-2$) is direction changing relative to  $\Theta_{i+1}$ whenever  $P_{i-1}\sim_N P_{i+1}$.
\end{defi}
	\begin{theorem}\cite[Theorem 3.5]{seemafnpart1}\label{main1} Let $p^l$ be a fixed natural number, where $p$ is a prime and $l\in \N$.
		\begin{enumerate}
			
			\item Suppose $x\in\mX_{p^l}$.  Then every well directed path defines a finite $\f_{p^l}$-continued fraction of $x$.
			\item The value of every finite $\f_{p^l}$-continued fraction belongs to $\mX_{p^l}$ and the continued fraction defines a well directed path in $\f_{p^l}$ from $\infty$ to its value with the convergents as  vertices in the path.
			\item 	Every real number has an $\mathcal{F}_{p^l}$-continued fraction expansion.
		\end{enumerate}
	\end{theorem}
\section{ $\f_{p^l}$-Continued Fractions with maximum +1}
\begin{defi}
	Suppose $ R_1,R_2\in\mX_{p^l}$ are such that $R_1\sim_{p^l} R_2$ in $\f_{{p^l}}$, where $R_i=r_i/s_i$ with $\mathrm{gcd}(r_i,s_i)=1$ for $i=1,2$. Then $R_1\oplus R_2$ denotes a rational number  $r/s$ where $r=r_1+r_2$ and $s=s_1+s_2$ and $R_2\ominus R_1$ denotes a rational number  $r'/s'$ where $r'=r_2-r_1$ and $s'=s_2-s_1.$ Operations $\oplus$ and $\ominus$ are referred to as \textit{Farey sum}\index{Farey sum} and \textit{Farey difference}\index{Farey difference} of two rational numbers. 
\end{defi}

\begin{lemma}\label{twoways}
	Let $P,Q\in\mX_{p^l}$ be two adjacent vertices in $\f_{p^l}$ and let
	$$P= P_1\to P_2\to\cdots\to P_{n+1}=Q$$ be a path in $\f_{p^l}$ such that $q_i<q_{i+1},~1\le i\le n,$ where $P_i=p_i/q_i,~\mathrm{gcd}(p_i,q_i)=1.$ Then $n\le 2$ and if $n=2,$  then $P_2=Q\ominus P.$
\end{lemma}
Since every element of $\mX_{p^l}$ has finite $\f_{{p^l}}$-continued fraction expansions, we have the following definition.

\begin{defi}Suppose $x\in\mX_{p^l}.$ An $\f_{p^l}$-continued fraction of $x$ not ending with $1/1$ is said to be an {\it $\f_{p^l}$-continued fraction with maximum $+1$} if it has maximum number of positive partial numerators excluding $\ep_1$, the first partial numerator, among all its $\f_{p^l}$-continued fraction expansions.
	
	An infinite $\f_{p^l}$-continued fraction
	$$\frac{1}{0+}~\frac{p^l}{b+}~\frac{\epsilon_{1} }{a_{1}+}~\frac{\epsilon_{2}}{a_{2}+}~\cdots\frac{\epsilon_{n}}{a_{n}+}\cdots$$ is said to be an \textit{ $\f_{p^l}$-continued fraction with maximum $+1$} if 
	$$\frac{1}{0+}~\frac{p^l}{b+}~\frac{\epsilon_{1} }{a_{1}+}~\frac{\epsilon_{2}}{a_{2}+}~\cdots\frac{\epsilon_{i}}{a_{i}}$$
	is an  $\f_{p^l}$-continued fraction with maximum $+1$ of the $i$-th convergent unless $(\ep_i,a_i)=(1,1).$
	\end{defi} 
 
%
%

\begin{defi}
	The path associated to an $\mX_{p^l}$-continued fraction with maximum $+1$ is called a {\it well directed path with maximum direction changing edges}.
\end{defi}	

In the following subsections, we discuss uniqueness of $\f_{p^l}$-continued fractions with maximum $+1.$

\subsection{Uniqueness of $\f_{p^l}$-Continued Fractions for  $x\in\mX_{p^l}$ or $x\in \R\setminus\Q$ }

\vspace{2mm}
Suppose $P$ and $R$ are adjacent vertices in $\f_{p^l}$. Then for $k\ge1,$ $\underbrace{P\oplus\cdots\oplus P}_{k\textnormal{-times}}\oplus R$ is a rational number of the form $\frac{ku+r}{kv+s}$, where $P=u/v$ and $R=r/s$ with $\mathrm{gcd}(u,v)=1=\mathrm{gcd}(r,s).$ For $k\ge2,$ we denote it by $(\oplus_k P)\oplus R$. 
\vspace{2mm}

\begin{lemma}\label{k-thfareysum}
	Suppose $P,R\in\mX_{p^l}$ are adjacent in $\f_{p^l}.$ Then there exists a natural number $k<p$ such that $(\oplus_k P)\oplus R\not\in\mX_{p^l}$.
\end{lemma}
\begin{proof}
	Let $P=u/v$ and $R=r/s$ be in $\mX_{p^l}$, where $\mathrm{gcd}(u,p)=1=\mathrm{gcd}(r,p).$ Suppose $\bar{x}$ is the congruence class of $x$ modulo $p.$ Then $$\{\overline{r},\overline{u+r},\overline{2u+r},\dots, \overline{(p-1)u+r}\}=\{\overline{0},\overline{1},\dots, \overline{p-1}\}.$$
	Hence, there is a positive integer $k<p$ such that $\overline{ku+r}=\overline{0}$ and the result follows.
\end{proof}	

	\begin{proposition}\label{nonuniquehalfpoint}
	Let $x\in\mX_{p^l}$. If $x=\lfloor p^lx\rfloor/p^l\oplus (\lfloor p^lx\rfloor+1)/p^l$, where $\mathrm{gcd}(\lfloor p^lx\rfloor,p)=1=\mathrm{gcd}(\lfloor p^lx\rfloor+1,p)$. Then $\f_{p^l}$-continued fraction expansion of $x$ with maximum $+1$ is not unique. 
\end{proposition}
\begin{proof}
Let $p_i/q_i$ be a sequence of $\f_{p^l}$-convergents of a real number. Then $q_i<q_{i+1}$. Therefore, the only possible choices of  $\f_{p^l}$-continued fractions of 	$x$ are $\frac{1}{0+}~\frac{p^l}{\lfloor p^lx\rfloor+}~\frac{+1 }{2}$ and $\frac{1}{0+}~\frac{p^l}{(\lfloor p^lx\rfloor+1)+}~\frac{-1 }{2}$, which $\f_{p^l}$-continued fractions with maximum $+1.$
\end{proof}
	If $ x\in\R\setminus\Q$ or $x\in\mX_{p^l}$. Let $P_0=a/p^l$ and $Q_0=(a+1)/p^l$ be  vertices in $ \f_{p^l}$ such that $P_0<x<P_0\oplus Q_0$. Then there is a well directed path from $\infty$ to $x$ via $P_0$ (by Lemma \cite[Lemma 4.1 ]{seemafnpart1}.  In fact, we have the following:
\begin{lemma}\label{lemma_uniquemaximumflips}
	Suppose $ x\in\R\setminus\Q$ or $x\in\mX_{p^l}$. Let $P_0=a/p^l$ and $Q_0=(a+1)/p^l$ be  vertices in $ \f_{p^l}$ such that $P_0<x<P_0\oplus Q_0$.
Then there is a unique  well directed path from $\infty$ to $x$ via $P_0$ having  maximum direction changing edges. 
\end{lemma}
\begin{proof}	Suppose $x\in\R$ and $x\not\in\Q\setminus\mX_{p^l}$ and there are two well directed paths from $\infty$ to $x$ with maximum direction changing edges via $P_0$. Let $\{P_i\}_{i\ge0}$ and $\{P_i'\}_{i\ge0}$ be the corresponding sequences of vertices, respectively. 
Now suppose  $P_i=P_i',~0\le i\le k-1,$ and $P_k\ne P_k'$. Without loss of generality, we may assume that $P_{k-1}<P_k<P_k'$ so that $P_{k-1}<P_k\le P_{k-1}\oplus P_{k}'<P_k'.$ First, we claim that  $x\ne P_{k-1}\oplus P_{k}'.$ If $ P_{k-1}\oplus P_{k}'\not\in\mX_{p^l}$, then $x\ne P_{k-1}\oplus P_{k}'$ as $x\not\in\Q\setminus\mX_{p^l}$. If $P_{k-1}\oplus P_{k}'\in\mX_{p^l}$  and $x=P_{k-1}\oplus P_{k}'$, then 
	$$\infty\to P_0\to\cdots\to P_{k-1}\to x$$ is the only path with maximal direction changing edges. This contradicts our assumption. Now we claim that $P_k= P_{k-1}\oplus P_k'$. Suppose $P_k\ne P_{k-1}\oplus P_k'$. If $P_k<x<P_{k-1}\oplus P_k'$, then any path through $P_k'$ in the direction of $x$ is not well directed. Similarly, if $P_{k-1}\oplus P_k'<x<P_k'$, then there is no well directed path to $x$ through $P_k$. Thus,  $P_k= P_{k-1}\oplus P_k'.$ Now, observe that the following path 
	\begin{equation}\label{extraflips}
	\infty\to P_0'\to\cdots\to P_{k-1}'\to P_k'\to P_k\to P_{k+1}\to\cdots\to x,
	\end{equation}
	where $P_k= P_{k-1}\oplus P_k'$, is a path from $\infty$ to $x$ having two additional direction changing edges and hence,  not well directed. Therefore, $P_{k+1}=P_{k}\oplus P_{k}'=\oplus_2P_{k-1}\oplus P_k'$ and $P_{k+2}$ is direction retaining with respect to the path
	$$\infty\to P_0'\to\cdots P_{k-1}'\to P_k'\to P_k\to P_{k+1}.$$
	Now consider the path
	\begin{equation}\label{extraflips1}
	\infty\to P_0'\to\cdots P_{k-1}'\to P_k'\to P_{k+1}=(\oplus_2P_{k-1}\oplus P_k')\to P_{k+2}\to\cdots\to x
	\end{equation}
	having two additional direction changing edges. Again, the path is not well directed so that $P_{k+2}=\oplus_3 P_{k-1}\oplus P_k'$. By repeating this argument, we get $P_{k+i-1}=\oplus_i P_{k-1}\oplus P_k'$ for $1\le i\le p-1$ which contradicts Lemma \ref{k-thfareysum}. 
\end{proof}
\begin{remark}\label{remark_maximumdchange}	Suppose $P_0=a/p^l,P_0'=(a+1)/p^l\in \mX_{p^l}$ and $x\in\R$ are such that $P_0<x<P_0\oplus P_0'$. Let   $P_0\sim_{p^l} P$ and $P_0\sim_{p^l} P'$ be two consecutive edges emanating from $P_0$ such that $P'<x<P.$ Then the following statements are easy to observe:
	\begin{enumerate}
		\item For some positive integer $k,$ $P=(\oplus_k P_0)\oplus P_0'$ and there is a well directed path from $\infty$ to $x$ through $P_0'$ if and only if for each $i,$ $1\le i\le k,$ $(\oplus_i P_0)\oplus P_0'\in\mX_{p^l}.$
		\item If there is a well directed path from $\infty$ to $x$ through $P_0'$ then the path is via $P.$
	\end{enumerate}
	
\end{remark}

\begin{proposition}\label{prop_maxdchngP0}
	Suppose $P_0=a/p^l,P_0'=(a+1)/p^l\in \mX_{p^l}$ and $x\in\R$ are such that $P_0<x< P_0'$. Then the well directed path from $\infty$ to $x$ with maximum direction changing edges is via $P_0$ if and only if $P_0<x<P_0\oplus P_0'$.
\end{proposition}
\begin{proof}
	Let $P_0<x<P_0\oplus P_0'$. By Lemma \ref{lemma_uniquemaximumflips}, there is a unique well directed path from $\infty$ to $x$ via $P_0$ with maximum direction changing edges, and  assume that the path is given by
	\begin{equation}\label{pathfromP0}
\Theta\equiv\infty\to P_0\to P_1\to P_2\to\cdots\to x.
	\end{equation}
	Let  $k$ be the number of direction changing edges in this path excluding $P_0\sim_{p^l} P_1.$
	Now suppose there is a well directed path from $\infty$ to $x$ via $P_0'$. By Remark \ref{remark_maximumdchange}, the path via $P_0'$ is through $P_1$ and $P_0<x<P_1$ so that
	\begin{equation}\label{pathfromP0'}
\Theta'\equiv	\infty\to P_0'\to\cdots\to P_r'\to P_1\to P_2\to\cdots\to x
	\end{equation}
	is a well directed path with maximum direction changing edges through $P_0'.$ Since $P_0<x<P_1<P_0\oplus P_0'$ the edge $P_1\to P_2$ is direction changing relative to $\Theta_1$ in path \eqref{pathfromP0} and the edge $P_1\to P_2$ is direction retaining relative to $\Theta_{r+1}'$ in path \eqref{pathfromP0'}. Further, note that the edge $P_i'\to P_{i+1}'$ in the path \eqref{pathfromP0'} is direction retaining relative to $\Theta_{i}'$ for $0\le i \le r.$ Thus, the path \eqref{pathfromP0'} has $k-1$ direction changing edges, which is a contradiction. So a well directed path from $\infty$ to $x$ with maximum direction changing edges is only via $P_0.$
\end{proof}
 We summarize Lemma \ref{lemma_uniquemaximumflips}, Remark \ref{remark_maximumdchange} and Proposition \ref{nonuniquehalfpoint} and { \ref{prop_maxdchngP0} in the following theorem.
\begin{theorem}\label{uniquemaximumflips}Suppose $x\in\R\setminus\Q$ or $x\in\mX_{p^l}$ such that $x\ne\lfloor p^lx\rfloor/p^l\oplus (\lfloor p^lx\rfloor+1)/p^l$ $x\not\in\Q\setminus\mX_{p^l}.$ Then 
	 there is a unique  well directed path with  maximum direction changing edges from $\infty$ to $x$. Consequently, an $\f_{p^l}$-continued fraction expansion of $x$ with maximum $+1$   is unique.
\end{theorem}

 
%
 \begin{theorem}\label{pathwithmaximumflips}
 	Suppose $x\in\R.$ The well directed path $\infty\to P_0\to P_1\to\cdots\to P_n\to\cdots\to x$   in $\f_{p^l}$ with maximum direction changing edges can be obtained by the following steps:
 	\begin{enumerate}
 		\item $P_0\sim_{p^l}\infty$ is chosen so that $|P_0-x|$ is the least possible.

 		\item For $i\ge1,$ 
 		if $P, P'\in\mX_{p^l}$ are such that  $P_{i-1}\sim_{p^l} P,$ $P_{i-1}\sim_{p^l} P'$,  $P\sim_{p^l} P'$ and $P_{i-1}<P'<x<P$. Then $P_i=P.$ 
 	\end{enumerate}
 \end{theorem}
 
 Now we utilize the algorithm given in Theorem \ref{pathwithmaximumflips} for finding an $\f_{p^l}$-continued fraction of $x\in\R$ with maximum $+1$.
 \begin{corollary} \label{algoformaximumflips} Given any $x\in\R$, an $\f_{p^l}$-continued fraction expansion $$\frac{1}{0+}~\frac{p^l}{b+}~\frac{\epsilon_{1} }{a_{1}+}~\frac{\epsilon_{2}}{a_{2}+}~\cdots\frac{\epsilon_{n}}{a_{n}}\cdots$$ of $x$ with maximum $+1$ is obtained as follows:
 	$$b=\left\{
 	\begin{array}{ll}
 	\lfloor p^l x\rfloor, &\mbox{if } 							
 	
 	(\lfloor p^l x\rfloor+1,p)\neq1\\\\
 	\lfloor p^l x\rfloor+1, &\mbox{if } 							
 	(\lfloor p^l x\rfloor,p)\neq1\\\\
 	
 	\lfloor p^l x\rfloor,& \mbox{ if } (\lfloor p^l x\rfloor,p)=1=(\lfloor p^l x\rfloor+1,p)  \textnormal{ and }  x<	\frac{\lfloor p^l x\rfloor}{p^l}\oplus\frac{\lfloor p^l x\rfloor+1}{p^l}\\\\

 	\lfloor p^l x\rfloor+1, &  \mbox{ if } (\lfloor p^l x\rfloor,p)=1=(\lfloor p^l x\rfloor+1,p) \textnormal{ and } 	x>	\frac{\lfloor p^l x\rfloor}{p^l}\oplus\frac{\lfloor p^l x\rfloor+1}{p^l}.
 	\end{array}
 	\right.$$	
 	
  Set $y_1=p^lx-b$, 
 	\begin{enumerate}
 		\item   $\ep_{i}=\sign (y_i);$\\
 		\item \begin{enumerate}
 			\item Suppose $1/|y_i|\in\N$ then 
 			$$a_i=\left\{ \begin{array}{lr}
 			\frac{1}{|y_i|}-1, & 	\textnormal{if } \frac{1}{|y_i|}-1\not\equiv -\ep_i p_{i-2}p_{i-1}^{-1}\mod p\\\\
 			1/|y_i|, & \textnormal{otherwise}.
 			\end{array}\right.$$

 			\item Suppose $1/|y_i|\not\in\N$ then 
 			$$a_i=\left\{ \begin{array}{lr}
 			\lfloor (\frac{1}{|y_i|})\rfloor, & 	\textnormal{if } \lfloor (\frac{1}{|y_i|})\rfloor\not\equiv -\ep_i p_{i-2}p_{i-1}^{-1}\mod p\\\\
 			\lfloor (\frac{1}{|y_i|}+1)\rfloor, & \textnormal{otherwise}.
 			\end{array}\right.$$
 			
 		\end{enumerate}
 		
 		\item    $y_{i+1}=\frac{1}{|y_i|}-a_i$. 
 	\end{enumerate}
 	\end{corollary}
 	\subsection{Non-uniqueness of $\f_{p^l}$-Continued fraction of $x\in\Q\setminus\mX_{p^l}$}
 In this subsection, we show that an $\f_{p^l}$-continued fraction expansion with maximum $+1$ of an element of $\Q\setminus\mX_{p^l}$  is not unique. In fact,  there are exactly two such expansions.
 	
  Observe that 
  	$x\in\Z$ if and only if $x=\frac{p^l\lfloor{x}\rfloor+1}{p^l}\oplus \frac{p^l\lfloor{x}\rfloor-1}{p^l}.$ For $0\le k\le l-1,$ set $$\frac{1}{p^k}\dot{\Z}=\{a/p^{k}: a\in\Z \textnormal{ and } \mathrm{gcd}(a,p)=1\}.$$ 
  	For $0\le k\le l-1,$
  		$x=a/p^k\in\frac{1}{p^k}\dot{\Z}$ if and only if $x=p^{l-k}a/p^l=\frac{p^{l-k}a+1}{p^l}\oplus \frac{p^{l-k}a-1}{p^l}.$	
  		
  	We adopt $\frac{1}{p^k}\dot{\Z}$ as a notation for the set $\{a/p^k:a\in\Z \textnormal{ and } \mathrm{gcd}(a,p)=1\}$ to differentiate it from $\frac{1}{p^k}{\Z}=\{a/p^k:a\in\Z\}.$
  	Here we record a lemma to generalize the above observations for each rational number which is not in $\mX_{p^l}.$
  \begin{lemma}\label{fareytwoways}
  	Let $a/b$ and $c/d$ be adjacent vertices in the Farey graph with $0<d<b.$ Suppose a reduced rational $x/y$ is adjacent to $a/b$ with $0<y<b$ then either $x/y=c/d$ or $x/y=(a-c)/(b-d).$
  \end{lemma}	
  \begin{lemma}\label{uniquefareysum} 
  	Suppose $x=r/(p^ks)\in\Q\setminus\mX_{p^l}$ with $p\nmid s$, $\mathrm{gcd}(r,p^ks)=1$ and $0\le k<l$. Then there exists a unique pair of vertices $R_1,R_2\in\mX_{p^l}$ such that $R_1\sim_{p^l} R_2$, and $r/(p^ks)=R_1\oplus R_2$, where  $R_i=r_i/(p^ls_i)$ with $\mathrm{gcd}(r_i,p^ls_i)=1$ for $i=1,2$ and  $0<s_1\le s_2$.
  \end{lemma}
  \begin{proof}
 It is enough to show that the following system has a solution
  	\begin{eqnarray*}
  	r_1+r_2&=& p^{l-k}r\\
  	s_1+s_2&=& s
  	\end{eqnarray*}
such that $	r_1s_2-r_2s_1=\pm1.$ Since $\mathrm{gcd}(p^{l-k}r,s)=1$, there exist integers
 $t,m,$ with $0<m<s$ so that $mp^{l-k} r-t s=\pm1.$ If $m\le s-m$, then set $r_1=t$ and $s_1=m$, else set $r_1=p^{l-k}r-t$ and $s_1=s-m.$
 The uniqueness follows from Lemma \ref{fareytwoways} and the fact that $\f_{p^l}$ is a subgraph of the Farey graph.
  \end{proof}
\begin{notation}
	For $x\in \Q\setminus \mX_{p^l}$, we will continue to use the notation $R_1, R_2$ with the properties stated in Lemma  \ref{uniquefareysum}. Let $N_x\in\N\cup\{0\}$ be such that the distance of $R_1$ from $\infty$ along the well directed path  with maximum direction changing edges is $N_x+1$.
\end{notation}

 By putting $s=2$ in Lemma \ref{uniquefareysum}, we get the following corollary:
  \begin{corollary}\label{coro_halfinteger}
  	Let $p>2$ be a prime and  $\frac{1}{2p^k}\dot{\Z}=\{a/(2p^{k}):a\in\Z \textnormal{ and }\mathrm{gcd}(a,2p)=1\}$ where $0\le k<l$.  Suppose $x\in\frac{1}{2p^k}\dot{\Z}$, then there exists a unique integer $t$ such that $x=t/p^l\oplus(t+1)/p^l.$
  \end{corollary}
  Now set \begin{equation}\label{Bp}\mathcal{B}_{p^l}=\left\{
  \begin{array}{ll} 
  \cup_{k=0}^{l-1}\frac{1}{2^k}\dot{\Z}, & \textnormal{ if } p=2\\\\
  
  \cup_{k=0}^{l-1}(\frac{1}{p^k}\dot{\Z}\cup\frac{1}{2p^k}\dot{\Z}), & \textnormal{ if } p\ne2.
  \end{array}\right.
   \end{equation}
  
 \begin{corollary}\label{dist_B_p}
If $x\in\mathcal{B}_{p^l},$ then $N_x=0.$	
 \end{corollary}
  \begin{proposition}\label{forNx}
  For $0\le k\le l-1,$	suppose $x=r/(p^ks)\in\Q\setminus\mX_{p^l}$ and $x=R_1\oplus R_2$. Let 
  	$\infty\to P_0\to P_1\to\cdots \to P_{N_x-1}\to P_{N_x}=R_1$ be the  well directed path from $\infty$ to $R_1$ with maximum direction changing edges, 
   then the  well directed path from $\infty$ to $R_2$ with maximum direction changing edges is $$\infty\to P_0\to P_1\to\cdots \to P_{N_x-1}\to P_{N_x}=R_1\to R_2.$$
   Further, 	the path from $\infty$ to $R_2$ extends to a well directed path from $\infty$ to $x.$ 
  \end{proposition}

  \begin{remark}  \label{longest cf1}  We have seen that $x\in\R$ has a finite $\f_{p^l}$-continued fraction if and only if $x\in\mX_{p^l}.$  Suppose $x\in\Q\setminus\mX_{p^l}$ so that its $\f_{p^l}$-continued fraction is infinite. The value of the infinite continued fraction $\frac{1}{2+}~\frac{-1}{2+}~\frac{-1}{2+}\cdots$ is 1. Here, we write two infinite $\f_{p^l}$-continued fraction expansions of $x\in\Q\setminus\mX_{p^l}.$ Suppose $x\not\in\mathcal{B}_{p^l}$. Then $N_x>0$ and the $\f_{p^l}$-continued fractions with maximum $+1$ of $R_1$ and $R_2$ are given by
  
  $$	R_1= \frac{1}{0+}~\frac{p^l}{b+}~\frac{\epsilon_{1} }{a_{1}+}~\frac{\epsilon_{2}}{a_{2}+}~\cdots\frac{\epsilon_{N_x}}{a_{N_x}} {\text{ and }}
  	R_2= \frac{1}{0+}~\frac{p^l}{b+}~\frac{\epsilon_{1} }{a_{1}+}~\frac{\epsilon_{2}}{a_{2}+}~\cdots\frac{\epsilon_{N_x}}{a_{N_x}+}~\frac{\ep_{N_x+1}}{a_{N_x+1}},\\
  $$
  	so that the following expressions are  $\f_{p^l}$-continued fractions of $x$
  	
  	\begin{subequations}\label{rationalnotinX}
  		\begin{equation}\label{rational1}
  		\frac{1}{0+}~\frac{p^l}{b+}~\frac{\epsilon_{1} }{a_{1}+}~\frac{\epsilon_{2}}{a_{2}+}~\frac{\ep_3}{a_3+}~\cdots\frac{\epsilon_{N_x}}{a_{N_x}+}~\frac{\ep_{N_x+1}}{a_{N_x+1}+y_{N_x+2}},
  		\end{equation}
  		\begin{equation}\label{rational2}
  		\frac{1}{0+}~\frac{p^l}{b+}~\frac{\epsilon_{1} }{a_{1}+}~\frac{\epsilon_{2}}{a_{2}+}~\cdots\frac{\epsilon_{N_x}}{a_{N_x}+}~\frac{\ep_{N_x+1}}{(a_{N_x+1}+2)-y_{N_x+2}},
  		\end{equation}
  			where $y_{N_x+2}=\frac{1}{2+}~\frac{-1}{2+}~\frac{-1}{2+}\cdots.$	 	\end{subequations} 	
  \label{longest cf2}
  	Suppose $x\in \frac{1}{p^i}\dot{\Z} $ for some $0\le i< l,$ then
  	\begin{subequations}
  		\begin{equation}\label{interger_cfr1}
  		x=\frac{1}{0+}~\frac{p^l}{(\lfloor{p^lx}\rfloor-1)+y};
  		\end{equation}
  		\begin{equation}\label{interger_cfr2}
  		x=\frac{1}{0+}~\frac{p^l}{(\lfloor{p^lx}\rfloor+1)-y},
  		\end{equation}
  	\end{subequations}
  	where $y=\frac{1}{2+}~\frac{-1}{2+}~\frac{-1}{2+}\cdots.$	
    	Similarly,  if $x\in \frac{1}{2p^i}\dot{\Z}$ for some $0\le i<l$ with  $p\ne2$, then
  	\begin{subequations}
  		\begin{equation}\label{halfinterger_cfr1}
  		x=\frac{1}{0+}~\frac{p^l}{\lfloor{p^lx}\rfloor+}~\frac{1}{3-y};
  		\end{equation}
  		\begin{equation}\label{halfinterger_cfr2}
  		x=\frac{1}{0+}~\frac{p^l}{(\lfloor{p^lx}\rfloor+1)+}~\frac{-1}{3-y},
  		\end{equation}
  	\end{subequations}
  	where $y=\frac{1}{2+}~\frac{-1}{2+}~\frac{-1}{2+}\cdots.$	
  \end{remark}

Using the fact that edges do not cross in $\f_{p^l}$, one can observe the following result.
  \begin{proposition}
  		Every $x\in\Q\setminus\mX_{p^l}$ has exactly two $\f_{p^l}$-continued fraction expansions  of $x$ with maximum $+1.$
  	\end{proposition}

  	\section{Best Approximations and Convergents}
  	Suppose $x$ is a real number. A reduced rational number $u/v$ is a best rational approximation of $x$ if $|vx-u|<|v'x-u'|$ for every $u'/v'\ne u/v$ with $0<v'\le v.$ 
  	 Best rational approximations of a real number are described by the convergents of the regular continued fraction. Here, we introduce best  $\mX_{p^l}$-approximations  of a real number.
  
  	\begin{defi}
  		A rational number $u/v\in \mX_{p^l}$ is called a {\it best $\mX_{p^l}$-approximation of $x\in\R$}, if for every $u'/v'\in\mX_{p^l}$ different from $u/v$ with
  		$0< v' \le v$, we have $|vx-u|<|v'x-u'|$.
  	\end{defi}

  	\begin{lemma}\label{sandwich}
  		Let $\{\frac{p_i}{q_i}\}_{i\ge0}$ be a  sequence of $\f_{p^l}$-convergents of $x\in\R$. If $u/v\in\mX_{p^l}$, then 
  		there exists a unique
  		solution $(\alpha,\beta)\in\Z\times\Z$ of the following system of equations 
  		\begin{equation}\label{systemofequations}
  		\begin{pmatrix} u\\v \end{pmatrix}= \alpha \begin{pmatrix} p_{n+1}\\q_{n+1}\end{pmatrix}+\beta\begin{pmatrix} p_n\\q_n\end{pmatrix}.
  		\end{equation}
  	\end{lemma}
  	\begin{proof} Since the determinant of the coefficient matrix of Eqs.\eqref{systemofequations} is $\pm p^l\ne0$ and the system has a unique solution. In fact, $\alpha=\pm (uq_n-vp_n)/p^l$ and
  		$\beta=\pm (vp_{n+1}-uq_{n+1})/p^l$. Note that $p^l$ divides $v,q_n$ and $q_{n+1}$ so that $(\alpha,\beta)\in\Z\times\Z$.
  	\end{proof}
  	
  	\begin{proposition}\label{appposition}
  		Let $x$ and $\{p_i/q_i\}_{i=0}^M$ be the sequence of its $\f_{p^l}$-convergents with maximum $+1.$ Then $u/v\in \mX_{p^l}$ with $q_{M-1}\le v<q_M$ is not its best $\mX_{p^l}$-approximation. 
  	\end{proposition}
  \begin{proof}
  	By Lemma \ref{sandwich},  $u/v=(p_{M}-p_{M-1})/(q_{M}-q_{M-1})$ (since $v>0,$ ). Thus, $|vx-u|=\frac{p^l}{q_M}=|q_{M-1}x-p_{M-1}|$. Note that $q_{M-1}<v$. Therefore, $u/v$ is not a best approximation.
  \end{proof}
  		\begin{proposition}\label{tail1}
  		Let $y_i$ be the $i$-th fin  of  $\f_{p^l}$-continued fraction of $x\not\in\Q\setminus\mX_{p^l}$ with maximum +1 . Then $|y_i|<1$ for every $i\ge1$.
  		\end{proposition}
  		\begin{proof}
  			By Theorem \ref{distinctconvergents}, for each $i\ge1,$ $|y_i|\le1$  and if $|y_i|=1$ for some $i\ge1$ then $x$ is a rational number. 
  		  By Corollary \ref{algoformaximumflips}, $y_{i+1}=1/|y_i|-a_i$ so that  $a_i=2$ for being $\f_{p^l}$-continued fraction with $+1$ which is infinite. We get a contradiction as $x\not\in\Q\setminus\mX_{p^l}.$
  		\end{proof}
  	
  		\begin{corollary}\label{tail2} 
  			 Let  $1/y_i=\pm2$ be the $i$-th fin  of the $\f_{p^l}$-continued fraction of $x\notin\Q\setminus \mX_{p^l}$ with maximum +1  for some $i\ge1$. Then $x\in\mX_{p^l}$. 
  		\end{corollary}
  		\begin{proof}
  			By Proposition \ref{tail1}, $y_{i+1}\ne\pm1$ and so $a_i=2.$ Thus, $y_{i+1}=0$ and $x\in\mX_{p^l}$. 
  		\end{proof}
  	
  		\begin{proposition}\label{rationalwithfareysum}
  			Suppose $x\in \Q\setminus(\mX_{p^l}\cup \mathcal{B}_{p^l}),$ where $\mathcal{B}_{p^l}$ is given by Equation \eqref{Bp}. 
  			If $y_i$ is the $i$-th fin of a  $\f_{p^l}$-continued fraction of $x$ with maximum +1  and $N_x$ is as in \eqref{rational1} and \eqref{rational2}, then  $y_i\neq\pm1$ and $y_{i-1}\ne\pm1/2$ for each $1\le  i\le N_x+1$ .
  			\end{proposition}
  			\begin{proof}We can see that $y_{N_x+1}\ne\pm1$ in each case.	Suppose $y_i=\pm1$ for some $i< N_x+1$. Then by Theorem \ref{algoformaximumflips}, $a_i=1$, (with $\ep_i=1$)  or $a_i=2$. If $a_i=1$, then $y_{i+1}=0$ which contradicts that $x\not\in\mX_{p^l}$. If $a_i=2$, then $y_{i+k}=-1, \forall k\ge1$  which contradicts that $y_{N_x+1}\neq\pm1$. The other part of the proposition can be proved by a similar argument.
  			\end{proof}

  		\begin{lemma}\label{bestappX} 
  			If $u/v\in\mX_{p^l}$ is  a best $\mX_{p^l}$-approximation of $r/(p^ls)\in\mX_{p^l}$, then $v\le p^ls$.
  		\end{lemma}
  		\begin{proof}
  			Suppose $v>p^ls$. Then $|vx-u|\ge0=|p^ls x-r|$ so that we get a contradiction. 
  		\end{proof}
  		\begin{lemma}\label{noapproximation}
  			Suppose $x\in \Q\setminus\mX_{p^l}$ with $x=R_1\oplus R_2$ as in Lemma \ref{uniquefareysum}. Let $u/v$ be a best $\mX_{p^l}$-approximation of $x$. Then $v\le p^ls_1$, where $R_1=r_1/(p^ls_1)$ and $R_2=r_2/(p^ls_2)$. Moreover, $x\in\mathcal{B}_{p^l}$ has no best $\mX_{p^l}$-approximation.
  		\end{lemma}	
  		\begin{proof} Suppose $x=r/(p^is)\in \Q\setminus\mX_{p^l}$ with $\mathrm{gcd}(r,s)=1=\mathrm{gcd}(r,p)$ and $0\le i\le l-1.$ 
  			Suppose $v>p^ls_1$. Observe that $p^{l-i}r={r_1+r_2}$ and $p^ls={p^ls_1+p^ls_2}$. Hence $|vx-u|\geq1/s=|p^ls_1 x-r_1|$ so that $u/v$ is not a best $\mX_{p^l}$-approximation of $x$.
  			
  		For $x=a/p^i\in\cup_{i=0}^{l-1}\frac{1}{p^i}\dot{\Z}$, the result follows by the observation: $p^lx=p^{l-i}a\in\Z$ so that $1=|p^l x -(p^lx-1)|=|p^l x - (p^lx+1)|.$ 
  			Similarly, let $x=a/(2p^i)\in\cup_{i=0}^{l-1}\frac{1}{2p^i}\dot{\Z},$ then $p^l x=(2m+1)/2$ for some non zero integer $m$ co-prime to $p$ so that $\lfloor p^lx\rfloor=m$. Thus the result follows by observing that $1/(2p^i)=|p^lx-\lfloor p^l x\rfloor |=|p^lx-(\lfloor p^lx\rfloor +1)|.$\end{proof}
  		\begin{example}
  			Let $x=1/5$. Then $x=\frac{4}{25}\oplus\frac{6}{25}$ in $\f_{25}$ so that $|25x-4|=1=|25x-6|.$ Therefore, $1/5$ has no best $\mX_{25}$-approximation.
  		\end{example}
  		
  		\begin{corollary}\label{constantdifference}
  			Suppose $x=r/s\in\Q\setminus\mX_{p^l}$ and $\{p_i/q_i\}$ is the sequence of convergents of one of the $\f_{p^l}$-continued fractions of $x$ with maximum $+1$. Suppose  $N_x$ is as in Remark \ref{longest cf1}. 
  			Then for $k\ge N_x$,	$|q_kx-p_k|=1/s.$
  		\end{corollary}

  		\begin{example}	Recall Example \ref{manyexpansions}, the  $\f_{5}$-continued fraction expansion of $11/40$ with maximum +1 is
  			$\frac{1}{0+}~\frac{5}{1+}~\frac{1}{2+}~\frac{1}{1+}~\frac{1}{2}$ and the corresponding convergents are given by
  			$$\Big\{\frac{1}{5},\frac{3}{10},\frac{4}{15},\frac{11}{40}\Big\}.$$
  			In fact, this list of convergents is the complete set of best $\mX_5$-approximations of $11/40$. 
  			
  		\end{example}	
  		\begin{example}
  			The $\f_5$-continued fraction expansion of $1/\pi$ with maximum +1 is
  				$$\frac{1}{0+}~\frac{5}{2+}~\frac{-1}{2+}~\frac{1}{2+}~\frac{1}{5}~\frac{-1}{2}~\frac{-1}{2+(\frac{355+113\pi}{78\pi-245})}$$
  			The first six $\f_5$-convergents are
  			$\Big\{\frac{2}{5},\frac{3}{10}, \frac{8}{25}, \frac{43}{135}, \frac{78}{245}, \frac{113}{355} \Big\}.$
  			One can verify that every element of this set is a best $\mX_5$-approximation of $1/\pi$.
  		\end{example}	
  \begin{example}
  The two $\f_5$-continued fraction expansions of ${7}/{27}$ with maximum $+1$  are 
  	$$\frac{1}{0+}~\frac{5}{1+}~\frac{1}{3+}~\frac{1}{2+}~\frac{1}{1+}~\frac{1}{3+y}
  \textnormal{ and }
  	\frac{1}{0+}~\frac{5}{1+}~\frac{1}{3+}~\frac{1}{2+}~\frac{1}{1+}~\frac{1}{1+}~\frac{1}{2+y},$$
  	where $y=\frac{-1}{2+}~\frac{-1}{2+}~\frac{-1}{2+}\cdots.$
  	 The corresponding convergents are
  	\begin{equation*}
  	\left\{\frac{1}{5},\frac{4}{15},\frac{9}{35},\frac{13}{50},\frac{48}{185},\cdots\right\}~\mathrm{and}~
  	\left\{\frac{1}{5},\frac{4}{15},\frac{9}{35},\frac{13}{50},\frac{22}{85},\cdots\right\}, 
  	\end{equation*}
  	respectively. Observe that $\{1/5,4/15,9/35,13/50\}$ is 
  	the set of common convergents of the two $\f_{5}$-continued fraction expansions of $7/27$. Again, we can see that these are the only best $\mX_5$-approximations of $7/27$. 
  \end{example} 	
  	We have a couple of theorems relating convergents of an $\f_{p^l}$-continued fraction expansion of a real number $x$ with maximum $+1$ and best 	$\mX_{p^l}$-approximations of $x$.

  		\begin{theorem}
  			Every best $\mX_{p^l}$-approximation of a real number is a convergent of its $\f_{p^l}$-continued fraction with maximum +1.
  		\end{theorem}
  		\begin{proof} 
  		Suppose $x\in\R.$	Let $\{P_i=\frac{p_i}{q_i}\}_{i\ge0}$ be a sequence of convergents of an $\f_{p^l}$-continued fraction of $x$ with maximum $+1$. Suppose $y_i$ is the $i$-th fin of the same $\f_{p^l}$-continued fraction of $x$. For convenience, set $x_i=1/|y_i|$.
  			Let $u/v\in\mX_{p^l}$ be a best $\mX_{p^l}$-approximation of $x$. Then for some $n \geq 1$,
  			$q_{n-1} \le v< q_{n}.$ If $x\in\mX_{p^l}$, by Lemma \ref{bestappX}, $v\le q_M$, where $x=p_M/q_M$. Observe that if $v=q_M,$ then $u=p_M$. If $v\ne q_M,$ then by Proposition \ref{appposition}, $v\le q_{M-1}$ so that $n\le M-1$. By Proposition \ref{tail1}, $x_{n+1}>1$ for $n\le M-1$. 
  			For $x\notin \mX_{p^l}$, we note that $x_i\ne1~\forall i\ge1$ unless $x\in\Q\setminus\mX_{p^l}$ and if $x\in\Q\setminus\mX_{p^l}$, then by Lemma \ref{noapproximation}, $N_x \geq 1$ and $n<N_x$ so that $x_{n+1}\ne1$ (Proposition \ref{rationalwithfareysum}). Thus, we can assume that $x_{n+1}>1$ in each case.
  			
  		 Suppose $u/v$ is not an $\f_{p^l}$-convergent of the  $\f_{p^l}$-continued fraction of $x$ with maximum $+1$.  By Lemma \ref{sandwich}, 
  			$v=\beta q_{n-1}+\alpha q_n, u=\beta p_{n-1}+\alpha p_n$ for $\alpha,\beta\in\Z$ with $|\beta|\ge1$ and by Theorem \ref{distinctconvergents}, 
  			\begin{eqnarray}
  				|vx-u|=\dfrac{p^l |\beta x_{n+1}-\ep_{n+1}\alpha|}{x_{n+1}q_n+\ep_{n+1}q_{n-1}}.\label{exp1}
  			\end{eqnarray}
  			Note that $|q_nx-p_n|=\dfrac{p^l }{x_{n+1}q_n+\ep_{n+1}q_{n-1}}$. We will show
  			that the numerator in \eqref{exp1} is strictly bigger than $p^l$ which contradicts that $\frac{u}{v}$ is a best $\mX_{p^l}$-approximation of $x$.
  			
  		\noindent	\textbf{Case 1.} Suppose $|\beta|=1$. Since $v>0,$  $u/v=(p_{n}-p_{n-1})/(q_{n}-q_{n-1})=P_{n}\ominus P_{n-1}.$ Note that the continued fraction is with maximum +1 and $u/v$ is not a convergent implies that $x$ lies between $P_n$ and $P_{n-1}$ so that $\ep_{n+1}=1$, which gives that
  			$$|\beta x_{n+1}-\ep_{n+1}\alpha|=|x_{n+1}+\ep_{n+1}|>1.$$
%
  			\textbf{Case 2.} Suppose $\beta\ge2$.
  			Then  $q_n \le v=\beta q_{n-1}+\alpha q_{n} < q_{n+1}=\epsilon_{n+1}q_{n-1}+a_{n+1}q_{n}$. Hence
  			$1-\beta < \alpha \leq a_{n+1}-1$ (since $q_{n-1}>0$). Using $|y_{n+2}|\le1$, $a_{n+1}\le x_{n+1}+1$, we have
  			$1-\beta < \alpha \leq x_{n+1}$.
  			These bounds on $\alpha$ imply $\beta x_{n+1}-\ep_{n+1}\alpha > 1$.
  			
  		\noindent	\textbf{Case 3.} Suppose $\beta\le -2$. Then
  			$0<v=\beta q_{n-1}+\alpha q_{n} < \epsilon_{n+1}q_{n-1}+a_{n+1}q_{n}$. Hence
  			$1\leq \alpha\le\ep_{n+1}-\beta+a_{n+1}-1$ (since $\frac{q_{n-1}}{q_n}<1$). 	These bounds on $\alpha$ imply $\beta x_{n+1}-\ep_{n+1}\alpha<-1$ unless  $\ep_{n+1}=-1$ and $\alpha=-\beta+a_{n+1}-2$. 
  			
  			Now,
  			suppose $\ep_{n+1}=-1$ and $\alpha=-\beta+a_{n+1}-2,$ then $u=\beta(p_n-p_{n-1})+(a_{n+1}-2)p_n$. Since $\ep_{n+1}=-1$,  $a_{n+1}\ge2.$ First, we consider $a_{n+1}\ge3$, then  $x_{n+1}\ge2$. If $x_{n+1}>2$ then  $\beta x_{n+1}-\ep_{n+1}\alpha<-1$. If $x_{n+1}=2$, then by Proposition \ref{tail2}, $x\in\mX_{p^l}$ and $P_{n+1}=x$ so that $a_{n+1}=2$ but we have considered that $a_{n+1}\ge3$. Thus inequality $\beta x_{n+1}-\ep_{n+1}\alpha<-1$ holds. Now consider the remaining case, $a_{n+1}=2,$ then $u/v=(p_{n+1}-p_n)/(q_{n+1}-q_n).$ Using a similar argument as in Case 1, we get the inequality.
  		\end{proof}
  		
  	\begin{proposition}\label{nobestappoffareysum}
  		Let $x\in\mX_{p^l}$. If $x=\lfloor p^lx\rfloor/p^l\oplus (\lfloor p^lx\rfloor+1)/p^l$, where $\mathrm{gcd}(\lfloor p^lx\rfloor,p)=1=\mathrm{gcd}(\lfloor p^lx\rfloor+1,p)$. Then $x$ has no best $\mX_{p^l}$-approximation other than itself.
  	\end{proposition}
  	\begin{proof}Suppose $r/s\in\mX_{p^l}$ is a best $\mX_{p^l}$-approximation of $x$. If $x\ne r/s$, then $s=p^l$. Observe that	$|p^lx-\lfloor p^lx\rfloor|=\frac{1}{2}=|p^lx-(\lfloor p^lx\rfloor+1)|$ and so $r/p^l$ is not a best $\mX_{p^l}$-approximation of $x.$	\end{proof}
  	\begin{example}
  		Take $3/10\in\mX_5$  with	$|5.3/10-1|=1/2=|5.3/10-2|$ so that $3/10$ has no best approximation other than itself.
  	\end{example}
  		\begin{theorem} Suppose $x\in\R$. Then
  			\begin{enumerate}
  				\item If $x\not\in\Q\setminus{\mX_{p^l}}$ and $x\ne\lfloor p^lx\rfloor/p^l\oplus (\lfloor p^lx\rfloor+1)/p^l$, then every convergent of the  $\f_{p^l}$-continued fraction of $x$ with maximum $+1$ is a best $\mX_{p^l}$-approximation of $x$.
  				
  				\item For $x\in\Q\setminus\mX_{p^l}$ or $x=\lfloor p^lx\rfloor/p^l\oplus (\lfloor p^lx\rfloor+1)/p^l$, an $\f_{p^l}$-convergent is a best $\mX_{p^l}$-approximation of $x$ if and only if it is a convergent of both  the  $\f_{p^l}$-continued fractions of $x$ with maximum $+1$.
  			\end{enumerate}
  		\end{theorem}

  		\begin{proof}			
  			Suppose $x\notin\Q\setminus \mX_{p^l}$ and $x\ne\lfloor p^lx\rfloor/p^l\oplus (\lfloor p^lx\rfloor+1)/p^l$.  Now let $\{\frac{p_k}{q_k}\}_{k=0}^M$ be the sequence of $\f_{p^l}$-convergents, where $M$ is finite if and only if $x\in\mX_{p^l}$.
  			Note that  $\frac{p_0}{q_0}=\frac{b}{p^l}$, where 	$b$ is given by Corollary \ref{algoformaximumflips}. This is clearly a best $\mX_{p^l}$-approximation of $x$.

  			Assume that, for $0 \le k \le n$, $p_{k}/q_{k}$  is a best $\mX_{p^l}$-approximation of $x$. Now we show that $p_{n+1}/q_{n+1}$ is a best $\mX_{p^l}$-approximation of $x$. When $M$ is finite and $n=M-1$, $p_M/q_M$ is a best $\mX_{p^l}$-approximation as $q_M x-p_M=0$. 	Thus, assume $n \geq 0$ is an integer with the restriction that $n<M-1$ when $M$ is finite.

  			For any $u/v\in\mX_{p^l}$ different from $p_{n+1}/q_{n+1}$ with $0<v\le q_{n}$, we have $|vx-u|>|q_{n}x-p_{n}|\ge|q_{n+1}x-p_{n+1}|$. Next assume $q_{n}< v\le q_{n+1}$. By Theorem \ref{distinctconvergents} (5),
  			\begin{align*}
			&|q_{n+1}x-p_{n+1}| =\frac{p^l}{x_{n+2}q_{n+1}+\epsilon_{n+2}q_{n}}. 
  			\end{align*}
  		By Lemma \ref{sandwich},
  			$u= \beta p_{n}+\alpha p_{n+1}$, $v= \beta q_{n}+\alpha q_{n+1}$ for some $\alpha,\beta\in\Z$. 
  			Thus,
  			\begin{align}|vx-u|=\frac{p^l| \beta x_{n+2}-\alpha \epsilon_{n+2}|}{x_{n+2}q_{n+1}+\ep_{n+2}q_n}.\label{scaledistance}\end{align}
  			Now, we will show that the numerator in \eqref{scaledistance} is greater than $p^l$. 	The proof of part (1) will be complete if we show that
  			\begin{equation}\label{greaterthan1}
  			| \beta x_{n+2}-\alpha \epsilon_{n+2}|>1.
  			\end{equation}
  			
  			\textbf{Case 1.} Suppose $|\beta|=1$. Then
  			$	| \beta x_{n+2}-\alpha \epsilon_{n+2}|=|x_{n+2}+\ep_{n+2}|,$
  			and $u/v=(p_{n+1}-p_n)/(q_{n+1}-q_n).$ The definition of well directed path with maximum direction changing edges forces that either $\ep_{n+2}=1$ or $P_{n+1}=x.$
  			Thus we have $|x_{n+2}+\ep_{n+2}|>1.$

  			\textbf{Case 2.} Suppose $\beta\ge2$. Since $v\le q_{n+1}$, we have
  			$(\alpha-1)q_{n+1}\le -\beta q_{n} <0$. Since $\alpha\in\Z$, we have
  			$\alpha\leq0$. Again, since $q_{n}< \beta q_{n}+\alpha q_{n+1}$, we have
  			$\frac{-\alpha}{\beta-1}< \frac{q_{n}}{q_{n+1}}$.
  			Hence, $\alpha> 1-\beta$ (since $q_{n}/q_{n+1}<1$). Thus we have shown $1-\beta<\alpha\le0$.
  			Using these bounds and the fact that $x_{n+2}> 1$ (by Proposition \ref{tail1}), inequality  \eqref{greaterthan1} follows.
  			
  			\textbf{Case 3.} Suppose $\beta\le -2$. Since $v>0$, $\alpha\geq 1$.
  			Since $v\le q_{n+1}$, $\frac{-\beta}{\alpha-1}\ge \frac{q_{n+1}}{q_{n}}$ so that $\alpha\le -\beta$ (since $q_{n+1}/q_{n}>1$). These bounds on $\alpha$ and $\beta$ implies inequality \eqref{greaterthan1} unless $\a=-\b$ and $\ep_{n+2}=-1$. Now suppose $\a=-\b$ and $\ep_{n+2}=-1$. Then $p/q=(p_{n+1}-p_n)/(q_{n+1}-q_n),$ we have discussed this possibility in Case 1. Thus we have $|x_{n+2}+\ep_{n+2}|>1.$

  			\medskip
  			To prove the second assertion of the theorem, let $x\in \Q\setminus \mX_{p^l}$. By Lemma \ref{noapproximation} and Proposition \ref{nobestappoffareysum}, if $x\in\mathcal{B}_{p^l}$ or $x=\lfloor p^lx\rfloor/p^l\oplus (\lfloor p^lx\rfloor+1)/p^l$, then $x$ has no best $\mX_{p^l}$-approximation. We assume that $x\ne\lfloor p^lx\rfloor/p^l\oplus (\lfloor p^lx\rfloor+1)/p^l$ $x\not\in\mathcal{B}_{p^l}$  so that $N_x\ge1$. 
  			Suppose $k\le N_x$. Then, by Proposition \ref{rationalwithfareysum}, we have $x_{k+1}>1$. Now the result follows by the same argument used in the first part.
  			
  		The	converse follows from Corollary \ref{constantdifference}.
  		\end{proof}
  		\bibliographystyle{plain}
  		\bibliography{bib}
  	\end{document}